\documentclass[a4paper,12pt,intlimits,oneside]{amsart}
\usepackage{graphicx}

\usepackage{enumerate}
\usepackage{amsfonts}
\usepackage{amsmath,amsthm,amssymb}

\newcommand{\comment}[1]{}
\numberwithin{equation}{section}
\newtheorem{theorem}{Theorem}[section]

\newtheorem{proposition}[theorem]{Proposition}
\newtheorem{lemma}[theorem]{Lemma}
\newtheorem{corollary}[theorem]{Corollary}
\newtheorem{problem}[theorem]{Problem}

\newtheorem{remark}[theorem]{Remark}

\newcommand\supp{\qopname\relax o{supp}}

\newcommand\s{\sigma}

\newcommand\de{\delta}

\newcommand\f{\varphi}

\newcommand\KK{\mathcal{K}}
\newcommand{\NN}{\mathbb N}
\newcommand{\ZZ}{\mathbb Z}
\newcommand{\RR}{\mathbb R}
\newcommand{\CC}{\mathbb C}
\newcommand{\TT}{\mathbb T}

\newcommand\CF{\mathcal{C}}

\newcommand\MM{\mathcal{M}}
\newcommand\Om{\Omega}

\newcounter{rek}
\newcounter{rev}
\newcounter{res}
\newcounter{rec}
\newcommand{\beql}[1]{\begin{equation}\label{#1}}

\newcommand{\eeq}{\end{equation}}

\newcommand{\ve}{\varepsilon}

\newcommand{\FF}{{\mathcal F}}

\begin{document}

\title[Carath\'eodory-Fej\'er problems on LCA groups]
{Carath\'eodory-Fej\'er type extremal problems \\ on locally compact Abelian groups}
\author{S\'andor Krenedits and Szil\'ard Gy. R\'ev\'esz}\thanks{Supported in part by the Hungarian National Foundation for Scientific Research, Project \#'s K-81658 and K-100461.
Work done in the framework of the project ERC-AdG 228005.
}

\date{\today}

\address{
\newline \indent A. R\'enyi Institute of Mathematics
\newline \indent  Hungarian Academy of Sciences
\newline \indent Budapest, Re\'altanoda utca 13--15.
\newline \indent 1053 HUNGARY
\newline and
\newline \indent Department of Mathematics
\newline \indent Kuwait University
\newline \indent P.O. Box 5969 Safat -- 13060 KUWAIT
} \email{krenedits@t-online.hu, revesz.szilard@renyi.mta.hu}

\begin{abstract}
We consider the extremal problem of maximizing a point value
$|f(z)|$ at a given point $z\in G$ by some positive definite
and continuous function $f$ on a locally compact Abelian group
(LCA group) $G$, where for a given symmetric open set $\Omega \ni z$, $f$ vanishes outside $\Omega$ and is normalized by $f(0)=1$. 

This extremal problem was investigated in $\RR$ and $\RR^d$ and
for $\Omega$ a 0-symmetric convex body in a paper of Boas and
Kac in 1943. Arestov and Berdysheva extended the investigation
to $\TT^d$, where $\TT:=\RR/\ZZ$. Kolountzakis and R\'ev\'esz
gave a more general setting, considering arbitrary open sets,
in all the classical groups above. Also they observed, that
such extremal problems occurred in certain special cases and in
a different, but equivalent formulation already a century ago
in the work of Carath\'eodory and Fej\'er.

Moreover, following observations of Boas and Kac, Kolountzakis
and R\'ev\'esz showed how the general problem can be reduced to
equivalent discrete problems of "Carath\'eodory-Fej\'er type"
on $\ZZ$ or $\ZZ_m:=\ZZ/m\ZZ$. We extend their results
to arbitrary LCA groups.
\end{abstract}

\maketitle

\vskip1em \noindent{\small \textbf{Mathematics Subject
Classification (2000):} Primary 43A35, 43A70. \\ Secondary
42A05, 42A82. 
\\[1em]
\textbf{Keywords:} Carath\'eodory-Fej\'er extremal problem,
pointwise Tur\'an problem, locally compact Abelian groups,
abstract harmonic analysis, Haar measure, convolution of
functions and of measures, positive definite functions,
Bochner-Weil theorem, convolution square, Fej\'er-Riesz
theorem.}

\section{Introduction}\label{sec:intro}

In this work we consider the following fairly general problem.

\begin{problem}\label{pr:ptwgeneral}

Let $\Omega\subset G$ be a given set in the Abelian group $G$
and let $z\in \Omega$ be fixed. Consider a \emph{positive
definite function} $f: G\to \CC$ (or $\to \RR$), normalized to
have $f(0)=1$ and \emph{vanishing outside of $\Omega$}. How
large can then $|f(z)|$ be?

\end{problem}

The analogous problem of maximizing $\int_\Omega f$ under the
same hypothesis was recently well investigated by several
authors under the name of "Tur\'an's extremal problem",
although later it turned out that the problem was already
considered well before Tur\'an, see the detailed survey
\cite{LCATuran}. The problem in our focus, in turn, was also
investigated on various classical groups (the Euclidean space,
$\ZZ^d$ and $\TT^d$ being the most general ones) and was also
termed by some as "the pointwise Tur\'an problem", but the
paper \cite{kolountzakis:pointwise} traced it back to Boas and
Kac \cite{BK} in the 1940's and even to the work of
Carath\'eodory \cite{Cara} and Fej\'er \cite{Fej} \cite[I, page
869]{Fgesamm} as early as in the 1910's.

So based on historical reasons to be further explained below,
this problem we will term as the \emph{Carath\'eodory-Fej\'er
type extremal problem on $G$ for $z$ and $\Om$}. This clearly
requires some explanation, since Carath\'eodory and Fej\'er
worked on their extremal problem well before the notion of
positive definiteness was introduced at all.

Positive definite functions on $\RR$ were introduced by
Matthias in 1923 \cite{Mathias}.
For Abelian groups positive definite functions are defined analogously \cite[p. 17]{rudin:groups} by the property that
\begin{equation}\label{eq:posdefdfnd}
\forall n\in \NN, ~\forall x_1,\dots,x_n\in G, ~ \forall c_1,\dots,c_n\in \CC \qquad \qquad \sum_{j=1}^{n}\sum_{k=1}^{n} c_j \overline{c_k} f(x_j-x_k) \geq 0.
\end{equation}
In other words, positive definiteness of a real- or complex
valued function $f$ on $G$ means that for all $n$ and all
choice of $n$ group elements $x_1,\dots,x_n\in G$, the $n\times
n$ square matrix $[f(x_j-x_k)]_{j=1,\dots,n}^{k=1,\dots,n}$ is
a positive (semi-)definite matrix. We will use the notation $f
\gg 0$ for a short expression of the positive definiteness of a
function $f:G\to \CC$ or $G\to \RR$.

Perhaps the most well-known fact about positive definite
functions is the celebrated \emph{Bochner theorem}, later
extended to locally compact Abelian groups (LCA groups for
short) in several steps and in this generality termed as the
\emph{Bochner-Weil theorem}. This states that a continuous
function $f:G\to \CC$ on a LCA group $G$ is positive definite
if and only if on the dual group $\widehat{G}$ there is an
essentially unique (positive) Borel measure $d\mu(\gamma)$ such
that $f$ is the inverse Fourier transform of $d\mu$:
$f(x)=\int_{\widehat{G}} \gamma(x) d\mu(\gamma)$ ($\forall x\in
G$), see e.g. \cite[page 19]{rudin:groups}.

We will not need this general theorem in its full strength, but
only the special case of positive definite sequences, obtained
actually in Carath\'eodory's and Fej\'er's time, preceding the
introduction and general investigation of positive definite
functions.
\begin{theorem}[\bf Herglotz]\label{p:Herglotz} Let $\psi:\ZZ\to
\CC$ be a sequence on $\ZZ$. Then $\psi \gg 0$ (i.e $\psi$ is
positive definite) if and only if there exists a positive Borel
measure $\mu$ on $\TT$ such that
\begin{equation}\label{eq:Herglotz}
\psi(n)=\int_{\TT} e^{2\pi i n t} d\mu(t) \qquad (n\in \ZZ).
\end{equation}
Furthermore, in case $\supp \psi \subset [-N,N]$ we have
$\psi\gg 0$ if and only if
$T(t):=\check{\psi}(t)=\sum_{n=-N}^{N} \psi(n) e^{2\pi i n t}
\geq 0$ ($t\in\TT$), and then $d\mu(t)=T(-t)dt$ and
$\psi(n)=\int_\TT T(t) e^{-2\pi i n t} dt$.
\end{theorem}
\begin{proof} The fact that any sequence represented in the
form of \eqref{eq:Herglotz} is necessarily positive definite
directly follows from the definition, as the reader can easily
check. (The same direct verification works in any LCA group,
too).

The existence of such a representation for an arbitrary
positive definite sequence on $\ZZ$ was first proved by
Herglotz in \cite{Herglotz}, preceding the later analogous
result of Bochner on $\RR$ and the development of the theory of
positive definite functions. The general proof on all LCA
groups belongs to Weil \cite{Weil}; for the proof and further
details see also \cite[1.4.3]{rudin:groups}.

The special case of finitely supported sequences can be fully
proved by a simple direct calculation.
\end{proof}

For further use we also introduce the extremal problems
\begin{align}\label{eq:CForigdef}
\MM(\Om)& :=\sup \{ a(1) ~:~ a:[1,N]\to \RR,~ N\in \NN,~ a(n)=0 ~ (\forall n\notin \Om),
\\ &\qquad\qquad \qquad\qquad T(t):=1+\sum_{n=1}^N a(n) \cos(2\pi nt)\geq 0 ~ (\forall t\in \TT)\}, \notag
\end{align}
which is called in \cite{kolountzakis:pointwise} the
\emph{Carath\'eodory-Fej\'er type trigonometric polynomial
problem} and
\begin{align}\label{CFfinitemdef}
\MM_m(\Om)& :=\sup \{ a(1) ~:~ a:\ZZ_m\to \RR,~ a(0)=1, ~a(n)=0 ~ (\forall n\notin \Om),
\\ &\qquad\qquad \qquad\qquad T\left(\frac{r}{m}\right):=\sum_{n \mod m} a(n) \cos\left(\frac{2\pi n r}{m}\right) \geq 0 ~ (\forall r \mod m) \}. \notag
\end{align}
which is termed in \cite{kolountzakis:pointwise} as the
\emph{Discretized Carath\'eodory-Fej\'er type extremal
problem}.

\begin{remark}\label{r:MmandM}
Obviously we have $\MM_m(\Om) \ge \MM(\Om)$, because the
restriction on the admissible class of positive definite
functions to be taken into account is lighter for the discrete
problem: we only need to have $T\left(\frac{r}{m}\right)\geq
0$, while for $\MM(\Om)$ the restriction is $T(t)\geq 0
~(\forall t\in \TT)$.
\end{remark}

Let us recall that Carath\'eodory and Fej\'er solved the
following extremal problem. Let $n\in \NN$ be fixed, and assume
that the 1-periodic trigonometric polynomial $T: \TT \to \RR$
of degree (at most) $n$ is \emph{nonnegative}. Under the
normalization that the constant term $a(0)=\int_\TT T=1$, what
is the possible maximum of $a(1)$ (solved already in
\cite{Cara}), and what are the respective extremal polynomials
(solved -- at all probability independently -- in \cite{Fej})?

Clearly the original Carath\'eodory-Fej\'er extremal problem is
a special case of the above $\MM(\Om)$ problem -- just take
$\Om:=[0,n]$, and observe that the possible odd part of $T$
(i.e. the sine series part of the trigonometric expansion) can
be neglected, for $a(1)$ is the same for a general $T(x)$ and
for $\frac12(T(x)+T(-x))$, the even part of $T$.

Let us now consider Problem \ref{pr:ptwgeneral} on $G:=\ZZ$,
with $\Om:=[-n,n]$, but with \emph{real valued} functions
(instead of general complex valued ones). Denote a function
from the admissible class (that is, a finite sequence of real
values on $[-n,n]$) as $\psi$ and assume that $\psi\gg 0$ on
$\ZZ$. As $\widehat{\ZZ}=\TT$, this is equivalent to say (in
view of Theorem \ref{p:Herglotz}) that the trigonometrical
polynomial $T(t):=\check{\psi}(t):=\sum_{k=-n}^n \psi(k)
\exp(2\pi i kt)$ is nonnegative, so also real. It follows that
$\overline{T}(t)=T(t)$, that is, $\psi(k)=\psi(-k)$. (Note that
positive definiteness of $\psi$ in itself implies that
$\psi(k)=\overline{\psi(-k)}$, as is seen from the general
introduction below, see \eqref{eq:fequalsftilde}, and so in
case $\psi$ is real-valued, we end up with the same relation).
Take now $a(0):=\psi(0),~ a(k):=2\psi(k)~(k=1,\dots,n)$. Then
the extremal problem translates to the $\MM([0,n])$ problem,
showing that \emph{for real valued functions} Problem
\ref{pr:ptwgeneral} on $\ZZ$ with $\Om:=[-n,n]$ is just the
same as the $\MM([0,n])$ problem (with a factor 2 between the
resulting extremal quantities). Below in Proposition
\ref{p:realpd} (ii) we show the easy fact that considering real
or complex valued functions does not matter (in this case of
sequences on $\ZZ$) -- therefore, we obtain that the original
Carath\'eodory-Fej\'er extremal problem is a (very) special
case of Problem \ref{pr:ptwgeneral}. This explains our
terminology.

\section{A short overview of basics about positive definite functions}\label{sec:posdefoverview}

Definition \ref{eq:posdefdfnd} has some immediate
consequences\footnote{These properties are basic and
well-known, see e.g. \cite[\S 1.4.1]{rudin:groups} We prove
them just for being self-contained, as they are easy.}, the
very first being that $f(0)\geq 0$ is nonnegative real (just
take $n:=1$, $c_1:=1$ and $x:=0$).

For any function $f:G\to \CC$ the \emph{converse}, or
\emph{reversed} function $\widetilde{f}$ (of $f$) is defined as
\begin{equation}\label{eq:ftildedef}
\widetilde{f}(x):=\overline{f(-x)}.
\end{equation}
E.g. for the characteristic function $\chi_A$ of a set $A$ we
have $\widetilde{\chi_A}=\chi_{-A}$ (where, as usual,
$-A:=\{-a~:~a\in A\}$), because $-x\in A$ if and only if $x\in
-A$.

Now let $f:G\to \CC$. Then in case $f$ is positive definite we necessarily have
\begin{equation}\label{eq:fequalsftilde}
f=\widetilde{f}.
\end{equation}

Indeed, take in the defining formula \eqref{eq:posdefdfnd} of
positive definiteness $x_1:=0$, $x_2:=x$ and $c_1:=c_2:=1$ and
also $c_1:=1$ and $c_2:=i$: then we get both $0\leq 2
f(0)+f(x)+f(-x)$ entailing that $f(x)+f(-x)$ is real, and also
that $0\leq 2f(0)+if(x)-if(-x)$ entailing that also
$if(x)-if(-x)$ is real. However, for the two complex numbers
$v:=f(x)$ and $w:=f(-x)$ one has both $v+w\in \RR$ and $i(v-w)
\in \RR$ if and only if  $v=\overline{w}$.

Next observe that for any positive definite function $f:G\to \CC$ and any given point $z\in G$
\begin{equation}\label{eq:posdefbdd}
|f(z)|\leq f(0),
\end{equation}
and so in particular if $f(0)=0$ then we also have $f\equiv 0$.
Indeed, let $z\in G$ be arbitrary: if $|f(z)|=0$, then we have
nothing to prove, and if $|f(z)|\ne 0$, let $c_1:=1$,
$c_2:=-\overline{f(z)}/|f(z)|$ and $x_1:=0$, $x_2:=z$ in
\eqref{eq:posdefdfnd}; then in view of \eqref{eq:fequalsftilde} $f(-z)=\overline{f(z)}$, which yields $0\leq 2f(0)+c_2 f(z)+\overline{c_2} f(-z)=2f(0)-2|f(z)|$ and
\eqref{eq:posdefbdd} obtains.

Therefore, all positive definite functions are bounded and
$\|f\|_\infty =f(0)$. That is an important property which makes
the analysis easier: in particular, we immediately see that the
answer to our extremal problem formulated in Problem
\ref{pr:ptwgeneral} cannot exceed 1.

Note also that similar elementary calculations show that
continuity of a positive definite function on a LCA group holds
if and only if the function is continuous at 0 c.f. \cite[(4),
p. 18]{rudin:groups} This we will not use, however.

For LCA groups, characters play a fundamental role, so it is of
relevance to mention that all characters $\gamma\in
\widehat{G}$ of a LCA group $G$ are positive definite. To see
this one only uses the multiplicativity of the characters to
get
$$
\sum_{j=1}^{n}\sum_{k=1}^{n} c_j \overline{c_k} \gamma(x_j-x_k)
= \sum_{j=1}^{n}\sum_{k=1}^{n} c_j\gamma(x_j) \overline{c_k \gamma(x_k)}
= \left|\sum_{j=1}^{n}c_j\gamma(x_j) \right|^2\geq 0
$$
for all choices of $n\in \NN$, $c_j\in \CC$ and $x_j\in G$
($j=1,\dots,n$). Similarly, for any $f\gg 0$ and character
$\gamma\in \widehat{G}$ also the product $f\gamma\gg 0$ since
for all choices of $n\in \NN$, $c_j\in \CC$ and $x_j\in G$
($j=1,\dots,n$) applying \eqref{eq:posdefdfnd} with
$a_j:=c_j{\gamma(x_j)}$ in place of $c_j$ ($j=1,\dots,n$) gives
\begin{align}\label{eq:charactertimesf}
\sum_{j=1}^{n}\sum_{k=1}^{n} c_j \overline{c_k} & \gamma(x_j-x_k) f(x_j-x_k)
= \sum_{j=1}^{n}\sum_{k=1}^{n} c_j\gamma(x_j) \overline{c_k \gamma(x_k)}f(x_j-x_k) \notag \\
&=\sum_{j=1}^{n}\sum_{k=1}^{n} a_j \overline{a_k} f(x_j-x_k)\geq 0 \quad \Big(a_j:=c_j\gamma(x_j) \quad (j=1,\dots,n)\Big).
\end{align}
It is equally easy to see directly from the definition that for
a positive definite function $f$ also $\overline{f} \gg 0$,
$f^{\sharp}(x):=f(-x)\gg 0$ and $\Re f\gg 0$, and that for
$f,g\gg0$ and $\alpha, \beta >0$ also $\alpha f + \beta g \gg
0$.

The perhaps most fundamental tool in topological groups is the
Haar measure, which is a non-negative regular and translation
invariant Borel measure $\mu_G$, existing and being unique up
to a positive constant factor in any LCA group, see \cite[p.
1,2]{rudin:groups} with a full proof.
As a direct consequence of uniqueness, we also have
$\mu_G(E)=\mu_G(-E)$ for all Borel measurable set $E$,
\cite[1.1.4]{rudin:groups}.

Following standard notations, in particular that of Rudin, we
simply write $dx, dy, dz$ etc. in place of $d\mu_G(x),
d\mu_G(y), d\mu_G(z)$ etc. Throughout the sequel we will
consider the convolution of functions with respect to the
Haar-measure $\mu_G$, that is
\begin{equation}\label{eq:convolutiondef}
(f\star g) (x) := \int_G  f(y)g(x-y) dy 
= \int_G  f(x+z)g(-z) dz
\end{equation}
defined for all functions $f,g \in L^1(\mu_G)$, or pairs of
functions $f\in L^p(\mu_G)$, $g\in L^q(\mu_G)$ with
$1/p+1/q=1$, see e.g. \cite[p. 3]{rudin:groups}. Convolution is
commutative and associative on any LCA group, see
\cite[1.6.1.Theorem]{rudin:groups}.

We will consider convolution of (bounded, complex valued,
regular Borel) measures and convolution of such measures and
functions as well. Rudin defines convolution of bounded regular
measures $\mu$ and $\lambda$ in \cite[1.3.1]{rudin:groups} with
reference to the product measure $\mu \times \lambda$ on
$G^2=G\times G$: to each Borel set $E\subset G$ the derived set
$E':=\{ (x,y)\in G^2~:~ x+y\in E\}$ is constructed and then
$\mu\star \lambda (E):= \mu \times \lambda (E')$. In particular
this also means that for $E\subset G$ a Borel set we have --
see \cite[(1) page 17]{rudin:groups} -- the formula
\begin{equation}\label{eq:convmeasuresofset}
\mu\star \lambda (E) = \int_G \mu (E-y) d\lambda(y).
\end{equation}
With this construction, convolution of any two (bounded,
regular, complex valued) measures is defined and yields another
such measure, moreover, convolution is commutative and
associative on any LCA group $G$, \cite[1.3.2
Theorem]{rudin:groups}. It is also easy to see, as is remarked
in \cite[(4), page 15]{rudin:groups}, that one can equivalently
define convolution of measures by the relation
\begin{equation}\label{eq:measureconv}
\int_G f d(\mu\star\lambda) := \int_G \int_G f(x+y) d\mu(x) d\lambda(y)
=\int_G \int_G f(x+y) d\lambda(y) d\mu(x) \,\, (f\in L^\infty(G)).
\end{equation}
Indeed, let the set of complex valued continuous functions with
compact support be denoted as $C_0(G)$: then, by the Riesz
representation theorem, the set $M(G)$ of all regular Borel bounded
(i.e. of finite total variation) measures is the topological
dual of $C_0(G)$ and thus can be written as $M(G)\cong
C_0^{*}(G)$. Now if $\mu, \lambda \in M(G)$, then their
convolution $\mu\star \lambda$ is defined according to
\eqref{eq:measureconv} for all $f\in C_0(G)$, which then
extends easily also to all $f\in L^\infty(G)$. Note that
\eqref{eq:convmeasuresofset} can be regarded as the special
case of $f=\chi_E$, for $\mu(E-y)=\int_G \chi_{E-y}(x)d\mu(x)
=\int_G \chi_{E}(y+x)d\mu(x)$.

Convolution of functions can then be regarded as a special case
of convolution of measures, for $f\star g$ is the density
function w.r.t. $\mu_G$ of the measure $\nu \star \s$ with $d
\nu:= f d\mu_G$ and $d\s := g d \mu_G$. Also convolutions of
measures with functions or functions with measures can be
obtained the same way. It is easy to see that for any $f\in
L^1(\mu_G)$ and $\nu\in M(G)$ we have the formula
\begin{equation}\label{eq:convmeasfunct}
f\star \nu (x) = \nu\star f (x) = \int_G f (x-y) d\nu(y).
\end{equation}
Another way to obtain this is to approximate $f$ by simple
functions and then use linearity and
\eqref{eq:convmeasuresofset} for each characteristic functions.
It is then immediate that the formula extends to $L^\infty(G)$,
too.

Also, analogously to \eqref{eq:ftildedef} the \emph{converse
measure} $\widetilde{\mu}(x):=\overline{\mu}(-x)$ (i.e.
$\widetilde{\mu}(E):=\overline{\mu}(-E)$) is defined to any
$\mu\in M(G)$. Then if $\phi \in C_0(G)$, then $ \int_G \phi
d(\widetilde{\mu\star\nu})= \int_G \phi(-x) d
\overline{(\mu\star\nu)}=\int_G\int_G\phi(-x-y)d
\overline{\mu}(x) d\overline{\nu}(y) =\int_G\int_G\phi(x+y)
d\widetilde{\mu}(x) d\widetilde{\nu}(y)=\int_G\phi
d(\widetilde{\nu} \star \widetilde{\mu})$, so that
$\widetilde{\mu\star\nu}=\widetilde{\mu}\star\widetilde{\nu}$.

For further use let us record here a few concrete formulae with
convolutions. By \eqref{eq:measureconv} 
for any $u,v\in G$ the formula $\de_u\star\de_v= \de_{u+v}$
holds true (where $\de_u \in M(G)$ denotes the Dirac measure
(unit point mass) at $u\in G$): for $ \int_G \phi
d(\de_u\star\de_v)= \int_G\int_G \phi(x+y)d \de_u(x) d\de_v(y)
= \phi(u+v)=\int_G \phi d\de_{u+v}$. Also, if $\phi \in
L^\infty(G)$ and $u\in G$, then we have in view of
\eqref{eq:convmeasfunct}
\begin{equation}\label{eq:deltastarphi}
\delta_u \star \phi (x) = \int_G \phi(x-y) d\de_u(y) = \phi(x-u).
\end{equation}
If for some Borel measurable $A$ we put $\phi:=\chi_A$, we
obtain similarly
\begin{equation}\label{eq:deltastarchi}
\de_u \star \chi_A (x) = \chi_A(x-u)=\chi_{A+u}(x).
\end{equation}

As \eqref{eq:convolutiondef} holds for all $L^1$ functions, it
also holds for $\chi_A, \chi_B$ with $A,B$ Borel measurable
sets with finite measure, yielding
\begin{equation}\label{eq:chistarchi}
\chi_A \star \chi_B (x) = \int_G \chi_A (y) \chi_B(x-y) dy 
= \int_G \chi_A  \chi_{(x-B)} d\mu_G = \int_G \chi_{A \cap (x-B)} d\mu_G = \mu_G (A \cap (x-B)).
\end{equation}
The same obtains also from calculating the measure convolution
$\mu_G|_A \star \mu_G|_B = \chi_A\mu_G \star \chi_B\mu_G$.

In this paper a particular role is played by the case of
$G=\ZZ$, where $\mu_\ZZ=\#$ is just the counting measure, and
thus all locally finite measures $\nu$ are absolutely
continuous and can as well be represented by their "density
function" $\f_\nu(k):=\nu(\{k\})$, and conversely, any function
$\f$ defines the respective measure $\nu_\f$ with $\nu_\f
(\{k\}):= \f(k)$, i.e. $\nu_\f=\sum_{k\in\ZZ} \f(k) \delta_k$;
moreover, clearly $L^1(\ZZ)=\ell^1\cong M(\ZZ)$. In particular
for $\f, \psi  \in \ell^1$, $\rho:=\f\star \psi$ is the density
function of the measure $\tau\in M(G)$ with $\tau=\nu\star \s$
and with $\nu:=\f_\nu d\#$, $\s:=\psi d\#$ being the measures
with density $\f, \psi$ respectively.

We will need the next well-known assertion (which we will use
only in the special case of compactly supported step functions,
however).

\begin{lemma}\label{l:convolutionsquare}
Let $f\in L^2(\mu_G)$ be arbitrary. Then the "convolution
square" of $f$ exists, moreover, it is a continuous positive
definite function, that is, $f\star \widetilde{f}\gg 0$ and
belongs to $C(G)$.
\end{lemma}
\begin{proof}
This can be found for LCA groups in \cite[\S 1.4.2(a)]{rudin:groups}.
\end{proof}

Although it is very useful when it holds, in general this
statement cannot be reversed. Even for classical Abelian
groups, it is a delicate question when a positive definite
continuous function has a "convolution root" in the above
sense. For a nice survey on the issue see e.g. \cite{EGR}. We
will however be satisfied with a very special case, where this
converse statement is classical.

\begin{lemma}\label{l:FR} (i) Let $\psi:\ZZ\to \CC$ be a finitely supported positive definite sequence. Then there exists another sequence $\theta:\ZZ\to \CC$, also finitely supported, such that $\theta*\widetilde{\theta}=\psi$. Moreover, if $\supp \psi \subset [-N,N]$, then we can take $\supp \theta \subset [0,N]$.

(ii) 
If $\psi:\ZZ_m\to\CC$, $\psi\gg 0$ on $\ZZ_m$, 
then there exists $\theta:\ZZ_m\to \CC$ with $\theta*\widetilde{\theta}=\psi$.
\end{lemma}

Note the slight loss of precision in (ii) -- it does not provide also localization, i.e. we cannot bound the support of $\theta$ in terms of a control of the support of $\psi$. This is natural, for the same finitely supported sequence can be positive definite on $\ZZ_m$ more easily than on $\ZZ$, as the equivalent restriction $T(2\pi n/m)\geq 0$ can be satisfied more easily than $T(t)\geq 0$ ($\forall t\in \TT$). However, here the support remains finite anyway, which is the only essential fact we need in our arguments below.
\begin{proof} [Proof of Part (i).]
Here we invoke the special case of Bochner's Theorem as
formulated in Theorem \ref{p:Herglotz} to get that
$T(t):=\check{\psi}(t)\geq 0$ ($\forall
t\in\TT=\widehat{\ZZ}$). Since $\psi$ is finitely supported,
$T$ is a $1$-periodic trigonometrical polynomial (with complex
coefficients $\psi(k)$).

Let $n$ stand for $\deg T$, so that $\supp \psi \subset [-N,N]$
translates to $n\leq N$. Write $T(t)$ in its trigonometric form
as $T(t)=\sum_{k=0}^N a_k \cos(2\pi kt) + b_k \sin(2\pi kt)$
with $a_k:=\psi(k)+\psi(-k)$ and $b_k=(\psi(k)-\psi(-k))/i$. A
glance at \eqref{eq:fequalsftilde} yields
$\psi(-k)=\overline{\psi(k)}$, so that then $a_k=2\Re \psi(k)$
and $b_k=2\Im \psi(k)$, whence in its trigonometrical form $T$
must have real coefficients $a_k, b_k \in \RR$ for all $0\leq k
\leq N$. This is of course obvious also from the usual
trigonometric version of the coefficient formulas:
$a_0=\int_\TT T(x)dx$, $b_0=0$, $a_k=2\int_\TT T(x)\cos(2\pi k
x) dx$, $b_k=2\int_\TT T(x)\sin(2\pi k x) dx$ ($k=1,\dots,N$).

Now the well-known classical theorem of L. Fej\'er and F.
Riesz, see \cite[Theorem 1.2.1]{Szego}, \cite{Fej}, or \cite[I,
page 845]{Fgesamm}, applies: there exists another
trigonometrical polynomial $P(t)$ of degree $n$ and with
complex coefficients -- more precisely, an algebraic polynomial
$p(z)$ of degree $n$ with $P(t)=p(e^{2\pi it})$ -- such that
$T(t)=|P(t)|^2$.

However, $|P|^2=P\cdot\overline{P}$ and by the well-known
properties of the Fourier transform, this means that there
exists a finitely supported $\theta:\ZZ\to \CC$, (the
coefficient sequence of $P$; whence actually it can be written
as $\theta:[0,n]\to \CC$, otherwise vanishing) such that
$\check{\theta}(t)=P(t)$ (and thus also
$\check{\widetilde{\theta}}=\overline{P}$) and
$\psi=\theta\star\widetilde{\theta}$. Note that $\supp \theta=[0,n]\subset [0,N]$, as needed.
\end{proof}
\begin{proof}[Proof of Part (ii).]
Consider $\widehat{\psi}(\nu):=\frac{1}{m} \sum_{j=0}^{m-1} \psi(j) \exp(-2\pi i \frac{j\nu}{m})$ which gives rise the representation $\psi(n)=\sum_{\nu\mod m} \widehat{\psi}(\nu) \exp(2\pi i \frac{n\nu}{m})$ (Fourier inversion on $\ZZ_m$).

First, observe that $\widehat{\psi}(\nu)\geq 0$ for all $\nu\in\ZZ_m$, for by definition \eqref{eq:posdefdfnd} of positive definiteness we must have with $x_j:=j\in\ZZ_m$ and $c_j:=\frac{1}{m}\exp(-2\pi i \frac{j\nu}{m})$ the inequality $0\leq \sum_{j\in\ZZ_m} \sum_{j'\in \ZZ_m} \psi(j-j') \frac{1}{m}\exp(-2\pi i \frac{j\nu}{m}) \frac{1}{m}\exp(2\pi i \frac{j'\nu}{m}) = \sum_{k\in\ZZ_m} \psi(k) \frac{1}{m} \exp(-2\pi i \frac{k\nu}{m}) =\widehat{\psi}(\nu)$.

Second, take $\widehat{\theta}(\nu):=\frac{1}{\sqrt{m}}\sqrt{\widehat{\psi}(n)} e^{i\f_\nu}$ ($\forall \nu\in\ZZ_m$), with arbitrary real $\f_\nu\in [-\pi,\pi)$. This gives rise to $\theta(n):=\sum_{\nu\in\ZZ_m} \widehat{\theta}(\nu) \exp(2\pi i \frac{n \nu}{m})$. Then we obtain
\begin{align*}
\theta\star\widetilde{\theta} (n) &:= \sum_{k\in \ZZ_m} \theta(k) \overline{\theta(k-n)} = \sum_{k\in \ZZ_m}  \sum_{\nu\in\ZZ_m} \widehat{\theta}(\nu) \exp(2\pi i \frac{k \nu}{m}) \sum_{\mu\in\ZZ_m} \overline{\widehat{\theta}(\mu)} \exp(2\pi i \frac{(n-k) \mu}{m})
\\&= \sum_{\nu\in\ZZ_m} \widehat{\theta}(\nu) \sum_{\mu\in\ZZ_m} \overline{\widehat{\theta}(\mu)} \exp(2\pi i \frac{n \mu}{m}) \sum_{k\in \ZZ_m} \exp(2\pi i \frac{k \nu}{m}) \exp(-2\pi i \frac{k \mu}{m})
\\&= \sum_{\nu\in\ZZ_m} |\widehat{\theta}(\nu)|^2 \exp(2\pi i \frac{n \nu}{m})~ m
= \sum_{\nu\in\ZZ_m} \widehat{\psi}(\nu) \exp(2\pi i \frac{n \nu}{m})=\psi(n).
\end{align*}
Clearly, here we have found a convolution squareroot $\theta$, but it is not guaranteed here that $\supp \theta \subset [0,N]$ or at least $[-N,N]$, even if $\supp \psi\subset [-N,N]$. On the other hand this can still suffice, as $\ZZ_m$, whence all supports, are a priori finite, hence compact.
\end{proof}

\section{Function classes and variants of the Carath\'eodory-Fej\'er type extremal problem}\label{sec:variants}

Already the above introductory discussion exposes the fact that
Problem \ref{pr:ptwgeneral} may have various interpretations
depending on how we define the exact class of positive definite
functions what we consider, and also on what topology we use on
$G$, if any (which determines what functions may be continuous,
Borel measurable, compactly supported, Haar summable, etc.).
Fixing the meaning of positive definiteness as in
\eqref{eq:posdefdfnd}, similarly to \cite{kolountzakis:groups},
in principle we may consider many different function classes
and corresponding extremal quantities. With respect to $f$
"living" in $\Omega$ only, three immediate possibilities are
that $f(x)=0~(\forall x\notin \Omega)$, that $\supp f \subset
\Omega$ and that $\supp f \Subset \Omega$ (the latter notation
standing for compact inclusion). For "nicety" of the function
$f$ one may combine conditions of belonging to $C(G)$
(continuous functions), $L^1(G)$ (summable functions)$,
L^1_{\rm loc}(G)$ (locally summable functions) etc.

In case of the analogous "Tur\'an problem" one maximizes the
integral $\int_G f d\mu_G$ rather than just a fixed point value
$|f(z)|$. In this question considerations of various classes
are more delicate, and although several formulations were shown
to be equivalent, see \cite[Theorem 1]{kolountzakis:groups},
the authors call attention to cases of deviation as well. In
the Carath\'eodory-Fej\'er extremal problem, however, we will
find that the solution is largely indifferent to any choice of
these classes, a somewhat unexpected corollary of our general
approach. So instead of formally introducing all kind of
function classes and corresponding extremal quantities, let us
restrict to the two extremal cases, that is the possibly widest
and smallest function classes, and define here only
\begin{eqnarray}
\label{Fjustpd} \FF_G^{\sharp}(\Om) &:=& \{f:G\to \CC~~:~~ f\gg 0, ~f(0)=1,~f(x)=0~\forall x\notin\Om\,\}~,
\\ \label{Fcontcompact} \FF^c_G(\Om) &:=& \{f:G\to \CC~~:~~ f\gg 0, ~f(0)=1,~ f\in C(G),~~\supp f
\Subset\Om~ \} \, .
\end{eqnarray}
Let us note, once again, that the first formulation is
absolutely free of any topological or measurability structure
of the group $G$. On the other hand, equipping $G$ with the
discrete topology the latter gives back a formulation close to
the former but with restricting $f$ to have finite support
only.

The respective "Carath\'eodory-Fej\'er constants" are then
\begin{equation}\label{CFconstants}
\CF_G^{\sharp}(\Om,z)~ := \sup \bigg\{|f(z)|\,:~ f \in \FF_G^{\sharp}(\Om) \bigg\},
\quad \CF_G^c(\Om,z)~ := \sup \bigg\{|f(z)|\,:~ f \in \FF^c_G(\Om) \bigg\}.
\end{equation}

In view of \eqref{eq:posdefbdd} giving that for $f\gg 0$
$\|f\|_\infty=f(0)$, the {\em trivial estimate} or
\emph{trivial (upper) bound} for the Carath\'eodory-Fej\'er
constants $\CF_G^{\sharp}(\Om,z)$ and $\CF^c_G (\Om,z)$ is thus
simply $f(0)=1$.

As for a lower estimation, in the most classical cases it is
easy to show that there exists a (real-valued) $f\in\FF^c_G(\Om)$
with $f(z)\geq 1/2$, so $\CF^c_G(\Om,z)\geq 1/2$. We will work
out this for the general case, too, in Proposition
\ref{p:lowerhalf} below, as later this may be instructive for
comprehending the proofs of our main results. However,
preceding it we discuss another issue.

By the above general definition, for $G=\ZZ$ and
$G=\ZZ_m:=\ZZ/m\ZZ$ the Carath\'eodory-Fej\'er constants
\eqref{CFconstants} with $z:=1$ -- and denoting by $H$ the fundamental set in place of $\Om$ in this case and writing $\FF_{\ZZ_m}(H):=\FF^\#_{\ZZ_m}(H)=\FF^c_{\ZZ_m}(H)$ -- are
\begin{align}\label{eq:CFOmZ}
\CF^{\#}(H)& :=\CF^{\#}_\ZZ(H,1):=\sup \{ |\f(1)|~:~ \f\in \FF_\ZZ^{\#}(H) \}
\notag\\& :=\sup \{ |\f(1)|~~:~~ \f: \ZZ\to \CC, ~ \f\gg 0,  ~\f(0)=1,~ \supp \f \subset H \},
\notag \\ \CF^c(H)& :=\CF^c_\ZZ(H,1):=\sup \{ |\f(1)|~:~ \f\in \FF^c_\ZZ(H) \}
\\& :=\sup \{ |\f(1)|~~:~~ \f: \ZZ\to \CC, ~ \f\gg 0, ~\f(0)=1,~ \supp \f \subset H, ~ \# \supp \f <\infty  \}, \notag \\
\CF_m(H)& :=\CF^{\#}_{\ZZ_m}(H,1)=\CF^c_{\ZZ_m}(H,1):=\sup \{ |\f(1)|~:~ \f\in \FF_{\ZZ_m}(H) \}
\notag \\& :=\sup \{ |\f(1)|~~:~~ \f: \ZZ_m\to \CC, ~ \f\gg 0, ~\f(0)=1,~ \supp \f \subset H \}.\notag
\end{align}
Similarly to discussion of various function classes, at this
point also discussion of the issue whether we consider
functions $f:G\to \CC$ or just real valued functions, occurs
naturally.

Note that in case of maximization of the integral $\int_\Omega
f$ in place of the single function value $|f(z)|$ (that is, in
case of the "Tur\'an problem") the paper
\cite{kolountzakis:groups} easily concludes that even in the
generality of LCA groups the restriction to real valued
functions does not change the extremal quantity. Indeed,
$S:=\supp f \Subset \Omega$ is always symmetric (for $f\gg 0$
implies $f=\widetilde{f}$) and so $\int_S f = \int_{(-S)}
\widetilde{f} = \int_S \overline{f}$, whence $\int_S f =\int_S
\Re f$, too.

However here, while extremalizing in various function classes
are generally easier to compare and remain equivalent, the
issue of real- or complex valued functions becomes more
interesting and in fact it splits in some cases while it
remains equivalent for others. In this preliminary section we
consider only the fundamental cases of $\ZZ$ and $\ZZ_m$ for
various $m\in \NN$. For a more concise notation first let us
write similarly to the complex valued case
\begin{align}\label{CFKKconstants}
\KK^{\#}_G(\Om,z):=\sup_{\f\in\FF_G^{\#\RR}(\Om)}|\f(z)|,
&\quad \KK_G^c(\Om,z):=\sup_{\f \in \FF_G^{c\RR}(\Om)} |\f(z)|,
\\
\KK^{\#}(H):=\KK^{\#}_{\ZZ}(H,1), \quad
\KK^{c}(H):=\KK^{c}_{\ZZ}(H,1), &\quad
\KK_m(H):=\KK_{\ZZ_m}(H,1):=\sup_{\f\in \FF_{\ZZ_m}^{\RR}(H)} |\f(1)|,\notag
\end{align}
where naturally we write for any group, (and so in particular
for $G=\ZZ$ and $G=\ZZ_m$)
$$
\FF_G^{\#\RR}(\Om):=\{\f: G\to \RR~:~\f\in \FF_G^{\#}(\Om)\},\qquad
\FF_G^{c\RR}(\Om):=\{\f: G\to \RR~:~\f\in \FF^c_G(\Om)\}.
$$
\begin{proposition}\label{p:realpd} We have the following.
\begin{enumerate}
\item[{\it (i)}] $\MM(H)=2 \KK^c(H)$ and for all $m\in \NN$ $\MM_m(H)=2 \KK_m(H)$.
\item[{\it (ii)}] We have $\KK^c(H)=\CF^c(H)$ and $\KK^{\#}(H)=\CF^{\#}(H)$. 
\item[{\it (iii)}] For all $m\in \NN$, $\cos(\pi/m) \CF_m(H) \leq \KK_m(H) \leq \CF_m(H)$.
\item[{\it (iv)}][\textbf{Ruzsa}] If $4 \leq m\in \NN$, then in general {\it (iii)} is the best possible estimate with both inequalities being attained for some symmetric subset $H\subset \ZZ_m$.
\item[{\it (v)}] If $m=2,3$, then for any admissible $H$ we must have $H=\ZZ_m$ and thus $\f(x)\equiv 1$ shows $\CF_m(H)=\KK_m(H)=1$.
\end{enumerate}
\end{proposition}
\begin{proof}
As regards {(i)}, $\MM(H)=2 \KK^c(H)$ is quite easy and
was already discussed in the final part of Section
\ref{sec:intro}. The analogous relation $\MM_m(H)=2
\KK_m(H)$ ($m\in \NN$) is seen the same way.

It remains to compare the respective extremal quantities for
the cases of real- and complex valued functions. The obvious
direction is that $\KK^c(H)\le  \CF^c(H)$, $\KK^{\#}(H)\le  \CF^{\#}(H)$; and also
$\KK_m(H) \leq \CF_m(H)$ for all $m\in \NN$.

For proving some estimate in the other direction, let now
$G=\ZZ_m$ or $\ZZ$ and $\psi\in \FF_G^c(H)$ be arbitrary. Let
further $\gamma_t\in \widehat{G}$ be the character belonging to
the parameter $t$, i.e. $\gamma_t(k):=\exp(2\pi i t k)$, where
$t\in \TT$ in case $G=\ZZ$, and $t:=j/m$ with $j\in\ZZ_m$ in
case $G=\ZZ_m$.

As said above, together with $\psi$, also $\psi\gamma_t\gg 0$
and even $\varphi:=\Re \{\psi \gamma_t\} \gg 0$--see \eqref{eq:charactertimesf} and around-- while belonging to the same function class $\FF^c_G(H)$, as we also have $\f(0)=\psi(0)$ and $\supp \f \subset \supp \psi=:S$. Thus
$\KK^c_G(H,1) \geq \sup_{\psi\in\FF^c_G(H), \gamma_t\in
\widehat{G}} \Re\{\psi(1)\gamma_t(1)\} \geq
\sup_{\psi\in\FF^c_G(H)} \min_{\alpha\in \TT} \\ \sup_{\gamma_t\in
\widehat{G}} \Re\{ |\psi(1)|e^{2\pi i \alpha} \gamma_t(1)\} =
\CF^c_G(H,1)\cdot \min_{\alpha\in \TT} \sup_{\gamma_t\in
\widehat{G}} \Re\{e^{2\pi i \alpha} \gamma_t(1)\}$.

With the choice of $t:=-\alpha$ this latter estimate gives for $G=\ZZ$ that $\KK^c(H) \geq \CF^c(H)$, and by a completely analogous computation with $\psi \in \FF_{\ZZ}^{\#}(H)$ we also find $\KK^{\#}(H) \geq \CF^{\#}(H)$. As the converse inequalities are obvious,
these furnish {(ii)}.

Furthermore, for $G=\ZZ_m$ we can always choose
$t:=-[m\alpha+1/2]/m$ and thus obtain $\KK_m(H) \geq
\min_{\alpha\in \TT} \cos\left(2\pi(\alpha-[m\alpha+1/2]/m)
\right)  \CF_m(H) = \cos(\pi/m) \CF_m(H)$, so also {(iii)}
obtains.

To find an example of equality $\KK_m(H)=\CF_m(H)$ is
trivial, as $H:=\ZZ_m$ suffices. To obtain the other extreme,
for $4\leq m\in \NN$ we use a construction communicated to us
by I. Z. Ruzsa. Namely, we take $H:=\{-1,0,1\}\subset \ZZ_m$,
compute the extremal quantity $\KK_m(H)$ and then
compare it to $\CF_m(H)$ as follows.

To start with, we prove $\KK_m(H)=1/2$ for an arbitrary
$m\geq 4$. First, for any positive definite real sequence
$\psi$ supported on $\{-1,0,1\}$ with $\psi(0)=1$, by
$\widetilde{\psi}=\psi$ we must have $\psi(-1)=\psi(1)$.
Second, take now $x_j:=j\quad (\mod m)$ and $c_j=(-1)^j
\quad(j=1,\dots,m)$. Then we will get from the definition
\eqref{eq:posdefdfnd} that $m-2m\psi(1)\geq 0$, so $\psi(1)\leq
1/2$. Third, the real sequence $1/2, 1, 1/2$ on $H$ is
positive definite according to Lemma \ref{l:convolutionsquare},
because it is the convolution square of the function
$\theta:\ZZ_m\to\RR$ defined as
$\theta(0):=\theta(1):=1/\sqrt{2}$ and $\theta(k):=0$ for all
$k \not\equiv 0,1 \mod m$.

Now let us find a lower estimate for the value of $\CF_m(H)$
for $m \geq 4$ \emph{even}. We consider the function
$\psi(0)=1$, $\psi(1)= r \exp(\pi i /m)$, where $r>0$ is a
parameter, and $\psi(-1)=\overline{\psi(1)}$, as is needed to
satisfy $\widetilde{\psi}=\psi$. We want $\psi\in
\FF_{\ZZ_m}(H)$ and $r$ maximal possible. Let us compute now
the Fourier transform $\widehat{\psi}(n)=\int_{\ZZ_m} \psi(k)
e^{2\pi i kn/m} d\mu_{\ZZ_m}(k)= \sum_{k \mod m} \psi(k)
e^{2\pi i kn/m}=1+r e^{\pi i (2n+1)/m} + r e^{-\pi i (2n+1)/m}=
1+2r\cos((2n+1)\pi /m)$. This remains nonnegative, for all $n
\mod m$ if and only if $r\leq 1/(2\cos(\pi/m))$; if $r$ equals
this bound, then for $n:=m/2~$ $\widehat{\psi}(m/2)= 0$. (Here
it is essential that $m$ is even!) So now we find that to keep
the Fourier transform, that is the scalar product with all
characters, nonnegative, it is necessary and sufficient that
$r\leq 1/(2\cos(\pi/m))$.

In fact it is a very special, trivial instance of the general
Bochner-Weil theorem that $\widehat{\psi}(n)\geq 0$ ($\forall n
\mod m$) is further equivalent to positive definiteness of the
sequence $\psi$ on $\ZZ_m$. For this particular case of the
general theorem let us note that characters are positive
definite, and so are their (finite) positive linear
combinations as said above, therefore also any function on
$\ZZ_m$ with nonnegative Fourier transform. To see the converse
statement, that is $\widehat{\psi} \geq 0$ if $\psi\gg 0$, we
can fix any $k \mod m$, take $n:=m$, $c_j:=e^{2\pi i jk/m}$,
$x_j:=j \mod m$ ($j=1,\dots,m$) and compute
$$
0\leq \sum_{j=1}^{m}\sum_{\ell=1}^{m} e^{2\pi i(j-\ell)k/m} \psi(j-\ell)
= m \sum_{a \mod m} e^{2\pi iak/m} \psi(a)=m \widehat{\psi}(k).
$$
So we thus find that $\CF_m(H)$ is at least the maximal $r$
in the above construction, which reaches $r=1/(2\cos(\pi/m))$.
It also follows that for $m$ even and at least 4,
$\CF_m(H)\geq \cos^{-1}(\pi/m) \KK_m(H)$.

Let now $m>4$ be \emph{odd}. For a similar construction as
above for even $m$, we now choose $\psi(0):=1$, $\psi(1):=r
\exp(2\pi i [m/2]/m)$ and consequently $\psi(-1):=r\exp(-2\pi i
[m/2]/m)$. Again, we use the Bochner characterization that
$\psi\gg 0$ on $\ZZ_m$ if (and only if) $\widehat{\psi}\geq 0$
on $\widehat{\ZZ_m}=\ZZ_m$. This means that for $k \mod m$ we
must have $0\leq \widehat{\psi}(k)=\sum_{\ell \mod m}
\psi(\ell) \exp(2\pi i k \ell/m)= 1 + r \exp(2\pi i (k
+[m/2])/m) +r \exp(-2\pi i (k +[m/2])/m)=1 + 2r \cos(2\pi(k
+[m/2])/m) = 1 - 2r \cos(\pi(2k-1)/m) $, which holds true for
all $k \mod m$ if and only if its minimum, with respect to $k$,
satisfies nonnegativity, that is, when $0\leq 1-2r \cos(\pi/m)$
and thus $r\leq 1/\{2\cos(\pi/m)\}$. This leads to exactly the
same estimate as before, that is, $\CF_m(H)\geq
\cos^{-1}(\pi/m) \KK_m(H)$. Therefore, {(iv)} obtains.

In view of $0,1\in H$ and $H$ being symmetric, for both $m=2$ and $m=3$ $H=\ZZ_m$ is clear, whence also {(v)} is obvious. The
Proposition is proved.
\end{proof}

Now we can formulate the already announced lower estimation
with the somewhat more precise form containing the same lower
estimate even with real functions.

\begin{proposition}\label{p:lowerhalf}
For any LCA group $G$, $z\in G$ and $0, \pm z\in \Om\subset G$
open set we have $\CF^c_G(\Om,z)\geq 1/2$, moreover, there exists
a \emph{real-valued} function $f\in \FF^{c\RR}_G(\Om)$ with
$f(z)\geq 1/2$.
\end{proposition}

\begin{proof}
Basically, we want to utilize the fact that the measure
$\nu:=2\delta_0 + \delta_z+\delta_{-z}$ is a positive definite
Borel measure. Instead of formally defining the notion of
positive definiteness of measures, let us remark here that
clearly $\nu=\sigma \star \widetilde{\sigma}$ with $\sigma :=
\de_0 + \de_z$ (and, as is easy to see, $\widetilde{\s} = \de_0
+ \de_{-z}$). From this starting point we then wish to
construct a positive definite, real-valued, continuous function
$F$, compactly supported within $\Om$, and with $F(z) \geq
\frac12 F(0)$.

As $0, \pm z \in \Om \subset G$ and $\Om$ is open, in the
locally compact group $G$ there exists an open set $U \ni 0, z,
-z$ with its compact closure $\overline{U} \Subset \Om$. Next
we take another open neighborhood $V$ of $0$ satisfying $V-V,
V-V-z, V-V+z \subset U$. Such a $V$ exists for all three
functions $(x,y)\to x-y$, $(x,y)\to x-y-z$, $(x,y)\to x -y+z$
are continuous from $G\times G$ to $G$ mapping $0$ to $0, -z,
z$, respectively, while all these images $0, \pm z$ lie in $U$.
(Or, saying it a bit differently: this is equivalent to $V-V
\subset U':=U\cap (U+z) \cap (U-z)$, which is still an open
neighborhood of $0$ and is thus such that there is $V$ with
$V\times V \to U'$ under $(x,y)\to x+y$.)

So formally with the characteristic function $\chi_V$ of $V$ we
now take $\Phi := \chi_V+\chi_{V+z}=\chi_V \star(\delta_0 +
\delta_z)$ and accordingly $\widetilde{\Phi}=
\widetilde{\chi_{V}} + \widetilde{\chi_{V+z}}
=\chi_{-V}+\chi_{-V-z}$, so that using \eqref{eq:chistarchi}
\begin{align*}
F(x):=\Phi\star\widetilde{\Phi}(x)= & (\chi_V+\chi_{V+z})\star(\chi_{-V}+\chi_{-V-z})(x)=
\\ = & \chi_V \star \chi_{-V}(x) + \chi_V \star \chi_{-V-z}(x) + \chi_{V+z}\star\chi_{-V}(x)+\chi_{V+z}\star\chi_{-V-z}(x)
\\ = & \mu_G(V\cap (x+V))+\mu_G(V\cap (x+z+V)) \\& +\mu_G((V+z)\cap (x+V))+\mu_G((V+z)\cap (x+z+V))
\\ = & 2\mu_G(V\cap (x+V))+\mu_G(V\cap (x-z+V))+\mu_G(V\cap (x+z+V)).
\end{align*}

By Lemma \ref{l:convolutionsquare}, $F\gg 0$, $F$ is
continuous, and obviously $\supp F \subset \supp \Phi + \supp
\widetilde{\Phi} = \overline{(V \cup (V+z))+((-V) \cup
(-V-z))}= \overline{(V-V)\cup (V-V-z) \cup (V-V+z)} \subset
\overline{U} \Subset \Om$ by construction, so $F\in C_0(G)$,
too and in fact $F/F(0) \in \FF^c_G(\Om)$. Moreover, $F$ is
real-valued, too.

Finally, denote $\alpha:=\mu_G(V)$, $\beta:=\mu_G(V\cap (z+V))$
and $\gamma:=\mu_G(V\cap (2z+V))$. Then $F(0)=2\mu_G(V) +
\mu_G(V\cap (V-z)) + \mu_G(V\cap(z+V))= 2\alpha + 2 \beta$, and
similarly $F(z)= 2\beta+\alpha + \gamma$. It follows that
$F(z)/F(0) =1/2 + (\beta + \gamma)/(2\alpha + 2 \beta) \geq
1/2$, as we wanted.
\end{proof}

Note that the construction also shows that if $o(z)=2$,
i.e. $2z=0$, then $\gamma=\alpha$ and $F(z)=F(0)=1$, i.e.
$\KK_G(\Om,z)=1$ taking into account the trivial estimate from
above, too.

We have noted in Proposition \ref{p:realpd} {(v)} that in
$\ZZ_2$, when $m=2$ (and thus in particular $o(1)=2$) and also in $\ZZ_3$, the
trivial choice of $f\equiv 1$ proves $\CF_{\ZZ_2}(H,z)=1$,
$\CF_{\ZZ_3}(H,z)=1$. Now we obtained also this in quite a
larger generality.

Next let us mention a continuity-type result.
\begin{proposition}\label{p:limiting} Let $H\subset \ZZ$ be a fixed symmetric \emph{finite} set containing 0 and 1. Then
\begin{equation}\label{eq:Cmlimit}
\lim_{m\to \infty} \KK_m(H) = \lim_{m\to \infty} \CF_m(H)= \CF^c(H).
\end{equation}
\end{proposition}
\begin{proof}
Consider first only the statement that $\lim_{m\to \infty}
\KK_m(H) = \KK^c(H)$, that is, restrict to
real-valued positive definite functions only. Note that even
the existence of the limit must be proved.

Since we deal with $m\to \infty$, we can assume $m>2\max H$.
Then obviously $\FF_{\ZZ}^{c\RR}(H) \subset \FF_{\ZZ_m}^{\RR}(H)$, hence $\KK^{c}(H) \leq \KK_m(H)$, as was remarked for the cosine formulation
already in Remark \ref{r:MmandM}. Whence $\KK^c(H) \leq
\liminf_{m\to\infty} \KK_m(H)$.

For an estimate from the other direction,let $\ve >0$ be
arbitrarily fixed,  and let $\f_m\in \FF_{\ZZ_m}^{\RR}(H)$ be such that $\f_m(1)\geq (1-\ve)\KK_m(H)$. We can then define the corresponding extension $\psi_m :\ZZ \to \RR$ with $\psi_m|_H=\f_m|_H$ and $\psi_m|_{\ZZ\setminus H}=0$. Then $\check{\psi}_m(k/m):=\sum_{j\in H}
\f_m(j) e^{2\pi i j k/m} \geq 0$ ($k \in \ZZ$) (even if
positive definiteness of $\f_m$ on $\ZZ_m$ does not imply
$\check{\psi}_m(t)\geq 0$ for all $t\in \TT$).

Now first we select a subsequence $(m_\ell)$ of the indices
with $\psi_{m_\ell}(1) \to \limsup_{m\to \infty}
\KK_m(H)$. By the Bolzano-Weierstrass Theorem we can
select a further subsequence $(m_{\ell_n}) \subset (m_\ell)$ --
which for simplicity we shall denote as $(m_n)$ from now on --
such that the function sequence $(\psi_{m_n})$ converges:
$\psi_{m_n}(j)\to \psi(j)$ for every $j\in H$ with some
finite value $\psi(j)$. (Here for the application of the
Bolzano-Weierstrass Theorem it is essential that
$\# H<\infty$, and also that by positive definiteness and normalization of $\f_m\in \FF_{\ZZ_m}^{\RR}(H)$ we necessarily have $|\psi_m(j)|=|\f_m(j)|\leq \f_m(0)=1$.)
Note that the limit values are also even ($\psi(-j)=\psi(j)$),
as by positive definiteness and having real values this
property holds for all $\f_m$ by \eqref{eq:fequalsftilde}.

Next let $\eta>0$ be arbitrary, and assume that for $n>N$ we
already have $|\psi_{m_n}(j)-\psi(j)|<\eta$. For an arbitrary
$t\in\TT$ let us choose a suitable $k_n \in \ZZ$ with
$|k_n/m_n-t|<1/m_n$: then we find
$$
\check{\psi}(t) =\sum_{j\in H} \psi(j) e^{2\pi i j t} \geq
\sum_{j\in H} \psi_{m_n}(j) e^{2\pi i j t} - \# H \eta \geq
\check{\psi}_{m_n}\left(\frac{k_n}{m_n}\right) - \# H
\frac{2\pi}{m_n} - \# H \eta
$$
because $|e^{2\pi it} - e^{2\pi is} |=|e^{2\pi i(t-s)}-1|= 2|\sin(\pi(t-s))|\le 2 \pi |t-s|$. However, $\f_m \gg 0$ (on $\ZZ_m)$ implies $\check{\psi_{m_n}}\left(\frac{k_n}{m_n}\right)=\check{\f_{m_n}}\left(\frac{k_n}{m_n}\right) \ge 0$, and thus we are led to
$$
\check{\psi}(t) \geq - \# H  \left( \frac{2\pi}{m_n} + \eta \right).
$$
Letting $n\to \infty$ and noting that $\eta>0$ was arbitrary
yields $\check{\psi}(t) \geq 0$, which shows $\psi \gg 0$ on
$\ZZ$ in view of Theorem \ref{p:Herglotz}. Furthermore, clearly
$\psi(0)=1$ and $\supp \psi \subset H$, hence $\psi\in
\FF_{\ZZ}^{c\RR}(H)$, while by construction $\psi (1) =
\lim_{n\to \infty} \psi_{m_n}(1) \ge (1-\ve) \limsup_{m\to
\infty} \KK_m(H)$ for any $\ve>0$.

So it follows that $\KK^{c}(H) \geq \limsup_{m\to \infty}
\KK_m(H) \geq \liminf_{m\to \infty} \KK_m(H)
\geq  \KK^{c}(H)$ (as recorded in the very first part),
furnishing $\lim_{m\to \infty} \KK_m(H) =
\KK^{c}(H)$. Note that according to Proposition
\ref{p:realpd} (ii), $\KK^{c}(H)=\CF^c(H)$. However, the
positive sequences $\KK_m(H)$ and $\CF_m(H)$ must be
equivalent regarding convergence in view of Proposition
\ref{p:realpd} (iii), so the first, and hence also the last
equality of \eqref{eq:Cmlimit} holds true.
\end{proof}

\begin{proposition}\label{p:finiteinfinite} We have
\begin{equation}\label{eq:finiteinfinite}
\CF^{\#}(H)=\CF^c(H).
\end{equation}
\end{proposition}
Therefore, taking into account also Proposition \ref{p:realpd} {(ii)} we can put $\CF(H):=\KK^c(H)=\KK^{\#}(H)=\CF^c(H)=\CF^{\#}(H)$ for any $H\subset \ZZ$.

\begin{proof}
Clearly the supremum is taken on a smaller set in $\CF^c(H)$,
hence $\CF^c(H)\leq \CF^{\#}(H)$.

Conversely, let $\f\in \FF^{\#}_\ZZ(H)$ and let us
consider the representation, given by Theorem \ref{p:Herglotz}
of Herglotz: $\f(n)=\int_\TT e^{2\pi i nt} d\nu(t)$, with $\nu$
a positive regular Borel measure on $\TT$.

Let $N\in \NN$ be arbitrary. Then we can consider
$\psi:=\psi_N:=\f \cdot \Delta_N$, where
$\Delta_N(n):=(1-\frac{|n|}{2N+1})_{+}$, and so in particular
$\Delta_N$ and $\psi_N$ have finite support.

First let us observe that $\Delta_N$ is positive definite. This
follows from Lemma \ref{l:convolutionsquare} writing
$$
\left(\chi_{[-N,N]} \star \chi_{[-N,N]}\right) (n) =\int_\ZZ
\chi_{[-N,N]}(n-j) \chi_{[-N,N]} (j) d\mu_\ZZ(j) =
\sum_{|j|, |n-j| \leq N} 1 = (2N+1)\Delta_N(n).
$$
Also, it easily follows from the fact that
$\Delta_N(n)=\check{F_N}(n)=\int_\TT e^{2\pi i nt} F_N(t) dt$,
where $F_N(t):=\frac1{2N+1} \left( \frac{\sin(\pi
(2N+1)t)}{\sin(\pi t)}\right)^2 \geq 0$ is the classical
Fej\'er kernel, providing the positive representation of
Herglotz described in Theorem \ref{p:Herglotz}. Note that this
means that $\Delta_N(n)$ is just the Fourier transform, (i.e.
the sequence of Fourier coefficients) of $F_N$.

Now the Herglotz-type positive representation for $\psi_N$
obtains from the usual rules of convolutions and the above:
$\psi_N(n)=\check{(F_N\star \nu)} (n)= \int_\TT e^{2\pi i nt}
\left(\int_\TT \frac1{2N+1} \left( \frac{\sin(\pi
(2N+1)(t-s))}{\sin(\pi (t-s))}\right)^2 d\nu(s)\right) dt$. That
is, $\psi_N:=\f \cdot \Delta_N\gg 0$, too.

Since now $\psi_N(1)=\f(1)(1-\frac{1}{2N+1})$, clearly
$\CF^c(H)\geq \sup_{N} \{|\psi_N(1)|\}=|\f(1)|$, and as this
holds for all possible $\f\in  \FF^{\#}_\ZZ(H)$, we get
$\CF^c(H)\geq \CF^{\#}(H)$ concluding the proof.
\end{proof}

Kolountzakis and R\'ev\'esz proves in \cite[\S 2, p.
404]{kolountzakis:pointwise} that in $\RR^d$ and for an
unbounded symmetric open set $\Om$ the bounded parts
$\Om_N:=\Om\cap B_N$, where $B_N=NB$ and $B\subset \RR^d$ is
the unit ball, approximate $\Om$ in such a way that
$\CF^c_{\RR^d}(\Om_N)\to \CF^c_{\RR^d}(\Om)$ as $N\to \infty$. (The
argument is essentially the same as the one above for
Proposition \ref{p:finiteinfinite}.) Analogously,
$\CF^c(\Om_N)\to \CF^c(\Om)$ also in the group $\ZZ$. These seem to
suggest that a limiting argument should give Proposition
\ref{p:limiting} even if $\Om\subset \ZZ$ is infinite. However,
this is false.
\begin{remark}
For any $\ve>0$ there exists an infinite set $H\subset \ZZ$,
sparse enough to have $\CF^c(H)\leq 1/2 + \ve$ but still
containing a copy of $\ZZ_m$ for every $m\in \NN$, and hence
having $\CF_m(H)=1$, the maximal possible value.

In fact, $H:=\{0,\pm 1,\pm N, \pm (N+1), \pm(N+2),\dots \}$
has $\CF^c(H)=1/(2\cos\frac{2\pi}{N+2})$, see \cite[Theorem 4.4
(iii)]{kolountzakis:pointwise}.
\end{remark}
This underlies the importance of carefully distinguishing
between the cases when we work in $\ZZ$ or in any $\ZZ_m$,
which explains why we formulated separately the two, otherwise
rather similar theorems in \S \ref{s:result}.

\section{Previous work on Carath\'eodory-Fej\'er type extremal problems}\label{sec:ptTuranintro}

For general domains in arbitrary dimension $d$ the problem was
formulated in \cite{kolountzakis:pointwise}. With our above
notations and general definition we can now recall it simply as
follows.
\begin{problem}[Boas-Kac - type pointwise extremal problem for
the space]\label{pr:pointwise} Find $\KK^c_{\RR^d}(\Omega,
z)$.
\end{problem}
\begin{problem} [Tur\'an - type pointwise extremal problem for the torus]\label{pr:torus}
Find $\KK^{c}_{\TT^d} (\Omega, z)$.
\end{problem}
As is easy to see, c.f. \cite[Remark
1.4]{kolountzakis:pointwise}, $\KK^c_{\TT^d} (\Omega, z)  \ge
\KK^c_{\RR^d} (\Omega, z)$, always.

The extremal value in the above Problem \ref{pr:pointwise} was
estimated together with its periodic analogue Problem
\ref{pr:torus} in the work \cite{ABB} in dimension $d=1$ for an
interval $\Omega:=(-h,h)$. Note that Boas and Kac have already
solved the interval (hence dimension $d=1$) case of Problem
\ref{pr:pointwise} in \cite{BK}, a fact which seems to have
been unnoticed in \cite{ABB}.

These problems are not only analogous, but also related to each
other, and, in fact, Problem \ref{pr:pointwise} is only a
special, limiting case of the more complex Problem
\ref{pr:torus}, see \cite[Theorem 6.6]{kolountzakis:pointwise}.
On the other hand, Boas and Kac have already observed, that
Problem \ref{pr:pointwise} (dealt with for $\RR$ in \cite{BK})
is connected to trigonometric polynomial extremal problems. In
particular, from the solution to the interval case they deduced
the value $\MM([0,n])=2\cos\frac{2\pi}{n+2}$ of the original
extremal problem due to Carath\'eodory \cite{Cara} and Fej\'er
\cite{Fej} or \cite[I, page 869]{Fgesamm}. They also
established a connection (see \cite[Theorem 6]{BK}) what
corresponds to the one-dimensional case of the first part of
\cite[Theorem 2.1]{kolountzakis:pointwise}.

Our results will extend these results together with the until
now most general results of \cite{kolountzakis:pointwise},
comprising all these and much more. So first let us record
these results here.
\begin{theorem}[Kolountzakis-R\'ev\'esz]\label{th:KolRev} In
$\RR^d$ and for any $z\in\RR^d$ and $\Om\subset \RR^d$ an open,
symmetric neighborhood of ${\bf 0}\in\RR^d$, we have with
$H(\Om,z):=\{k\in \ZZ~:~ kz\in \Om\}$ the relation
\begin{equation}\label{eq:KRRdcase}
\KK^{\#}_{\RR^d}(\Omega,z)=\KK^c_{\RR^d}(\Omega,z)=\KK(H(\Omega,z)).
\end{equation}
If $\Om\subset\TT^d$ is an open symmetric neighborhood of ${\bf
0}\in\TT^d$, and the order of $z$ is infinite (i.e. $z$ has no
torsion), then we have with $H(\Om,z):=\{k\in \ZZ~:~ kz\in
\Om\}$
\begin{equation}\label{eq:KRTTdozinfty}
\KK^{\#}_{\TT^d}(\Omega,z)=\KK^c_{\TT^d}(\Omega,z)=\KK(H(\Omega,z)).
\end{equation}
Finally, if the order of $z\in\TT^d$ is $o(z)=m$, then with
$H_m(\Om,z):=\{k\in \ZZ_m~:~ kz\in \Om\}$ we have
\begin{equation}\label{eq:KRTTdozm}
\KK^{\#}_{\TT^d}(\Omega,z)=\KK^c_{\TT^d}(\Omega,z)=\KK_m(H_m(\Omega,z)).
\end{equation}
\end{theorem}
Actually, the above can be collected from \cite[Theorem
2.1]{kolountzakis:pointwise} and \cite[Theorem
2.4]{kolountzakis:pointwise}. The most important aspect of it
is perhaps the understanding that the above point-value
extremal problems depend only on the set $H(\Om,z)$ and the
order of $z$ itself, and are in fact equivalent to the
trigonometric polynomial extremal problems given in
\eqref{eq:CForigdef} and \eqref{CFfinitemdef}. In other words,
the result is transferring information to the given more general problem from the corresponding equivalent other problem in $\ZZ$ or in $\ZZ_m$ in all cases.
Until that work the equivalence remained unclear in spite of
the fact that, e.g., Boas and Kac found ways to deduce the
solution of the trigonometric extremal problem
\eqref{eq:CForigdef} from their results on Problem
\ref{pr:pointwise}. Kolountzakis and R\'ev\'esz also obtained a
clear picture of the limiting relation between torus problems
and space problems, formulated above as Problem \ref{pr:pointwise} and
Problem \ref{pr:torus}, and parallel to this, between the
finitely conditioned trigonometric polynomial extremal problem
\eqref{CFfinitemdef} and the positive definite trigonometric
polynomial extremal problem \eqref{eq:CForigdef}. Furthermore,
the investigation was extended to arbitrary (symmetric open)
sets $\Omega \subset \RR^d$ or $\TT^d$, dropping the condition
of convexity of $\Omega$.

Let us remark, however, that even with the above equivalence result, the actual calculation of the extremal values may still take considerable work and innovation, see e.g. \cite{Ivanov}. For the numerous applications see the original paper \cite{kolountzakis:pointwise} and the references \cite{ABB, BK, R2, Mini, R3}.

Ending this section, let us recall that investigation of
so-called Tur\'an-type problems started with keeping an eye on
number theoretic applications and connected problems. The
interesting papers of Gorbachev and Manoshina \cite{GorbMan01,
GorbMan04} mention \cite{KS} and character sums; applications
to van der Corput sets were mentioned by several authors and in
particular by Ruzsa \cite{ruzsa:uniform}. Here we recall
another question of a number theoretic relevance, open for at
least two decades by now, and also mentioned in
\cite{kolountzakis:pointwise}.
\begin{problem}\label{Deltaproblem} Determine $ \Lambda (n) :=
\sup \{ \MM(H)/2 \, : \,\, 1\in H\subseteq \NN, |H|=n \}$.
\end{problem}
We only know (cf \cite{R1}) $ 1-{5 \over (n+1)^2} \le \Lambda
(n) \le 1 - {0.5 \over (n+1)^2}$. The question is relevant to
the Beurling theory of generalized primes, see \cite{R4}.

\section{Formulation of the main results}\label{s:result}

For points $z\in G$ with infinite order the problem becomes
equivalent to the trigonometric polynomial extremal problem of
the sort \eqref{eq:CForigdef}.

\begin{theorem}\label{th:KreneditsL}
Let $G$ be any locally compact Abelian group and let
$\Om\subset G$ be an open (symmetric) neighborhood of $0$. Let
also $z\in \Om$ be any fixed point with $o(z)=\infty$, and
denote $H(\Om,z):=\{k\in \ZZ~:~ kz\in \Om\}$. Then we have
\begin{align}\label{eq:mainresult}
\CF^c_G(\Om,z)=\KK_G^c(\Om,z)=\CF^\#_G(\Om,z)=\KK_G^\#(\Om,z)=\CF(H(\Om,z)).
\end{align}
\end{theorem}

\begin{corollary}\label{c:zinftyalleq} For $G$ any locally compact Abelian group,
$\Om\subset G$ any open (symmetric) neighborhood of $0$, and
$z\in \Om$ any fixed point with $o(z)=\infty$, we have
$\CF^c_G(\Om,z)=\KK_G^c(\Om,z)=\CF^\#_G(\Om,z)=\KK_G^\#(\Om,z)$,
the common value of which can thus be denoted simply by
$\KK_G(\Om,z)$ or $\CF_G(\Om,z)$.
\end{corollary}

If $z\in G$ is cyclic (has torsion), the situation is
analogous: then Problem \ref{pr:ptwgeneral} reduces to a
well-defined discrete problem of the sort \eqref{CFfinitemdef}.

\begin{theorem}\label{th:KreneditsFini}
Let $G$ be any locally compact Abelian group and let
$\Om\subset G$ be an open (symmetric) neighborhood of $0$. Let
also $z\in \Om$ be any fixed point with $o(z)=m<\infty$, and
denote $H_m(\Om,z):=\{k\in \ZZ_m~:~ kz \in \Om\}$. Then we have
\begin{equation}\label{eq:mainresultFini}
\CF_G^\#(\Om,z)=\CF_G^c(\Om,z)=\CF_m(H_m(\Om,z)) \quad \text{\rm and}\quad
\KK_G^{\#}(\Om,z)=\KK_G^{c}(\Om,z)=\KK_m(H_m(\Om,z)).
\end{equation}
\end{theorem}

\begin{corollary}\label{c:zfiniteeq} For $G$ any locally compact Abelian group,
$\Om\subset G$ any open (symmetric) neighborhood of $0$, and
$z\in \Om$ any fixed point with $o(z)<\infty$, we still have
$\CF^c_G(\Om,z)=\CF^\#_G(\Om,z)$ and
$\KK_G^c(\Om,z)=\KK_G^\#(\Om,z)$, the common value of which can
thus be denoted by $\CF_G(\Om,z)$ and $\KK_G(\Om,z)$,
respectively.
\end{corollary}
Note that this also holds true if $o(z)=\infty$, furthermore, than $\CF_G(\Om,z)=\KK_G(\Om,z)$ according to Corollary \ref{c:zinftyalleq}. We will use these notations in the last section.

\section{Proofs of the Main Results}\label{s:proof}

\begin{proof}[Proof of Theorem \ref{th:KreneditsL}]
We present the argument only for
$\CF_G^\#(\Om,z)=\CF_G^c(\Om,z)=\CF^c(H(\Om,z))$ (i.e. the
complex case), for the real variant
$\KK_G^{\#}(\Om,z)=\KK_G^{c}(\Om,z)=\KK^{c}(H(\Om,z))$ is
completely similar. Once these real and complex variants are
proved, a combination of Proposition \ref{p:realpd} (ii) and
Proposition \ref{p:finiteinfinite} gives that the right hand
sides of these equalities are all equal.

\smallskip

To simplify the notation somewhat, we will write throughout
this proof $H:=H(\Om,z)$.

As we trivially have $\CF_G^\#(\Om,z)\geq\CF_G^c(\Om,z)$, to
derive the equality of these quantities of the complex setup
and $\CF(H)$--which is the common value of
$\CF^\#(H)=\CF^c(H)$ in view of Proposition
\ref{p:finiteinfinite}--it suffices to prove $\CF_G^\#(\Om,z)
\leq \CF(H)(= \CF^\#(H)=\CF^c(H))\leq \CF_G^c(\Om,z)$ only.

\smallskip

So we are to prove only two inequalities, the first being that
$\CF_G^\#(\Om,z) \leq \CF^\#(H)$. Let us take any $f\in
\FF^\#_G(\Om)$, and consider the subgroup $Z:=\langle z \rangle
\leq G$.

Observe that $g:=f|_Z \gg 0$ on $Z$, for if the defining
requirements \eqref{eq:posdefdfnd} hold for all selections of
the $x_j\in G$, then obviously they must also hold for all
values chosen from $Z$. So this way we have defined a function
$g \in \FF^\#_Z((\Om\cap Z))$. Finally, let us remark that the
natural isomorphism $\eta:\ZZ\to Z$, which maps according to
$\eta(k):=kz$, carries over $g$, defined on $Z\le G$, to a
function $\psi:=g\circ \eta$, which is therefore positive
definite on $\ZZ$, has normalized value $\psi(0)=g(0)=f(0)=1$,
and $\supp \psi \subset H$ for $\supp g \subset \supp f \Subset
\Om$.

From here we read that $|f(z)| \leq \sup \left\{ |\psi(1)| ~:~
\psi \in \FF^\#_\ZZ(H) \right\} =\CF^\#(H)$. Taking $\sup_{f\in
\FF^\#_G(\Om)}$ on the left hand side concludes the proof of
the first part.

\smallskip

It remains to show the inequality $\CF^c(H)\leq
\CF^c_G(\Om,z)$.

Let $\psi\in \FF^c_\ZZ(H)$, so also positive definite and of
finite support $S\subset H$, say. We also define the respective
measure $\nu:=\sum_{k\in S} \psi(k) \delta_{kz}$, where $\de_s$
is the Dirac measure, concentrated at $s\in G$. By definition
of $H$, $S=\supp \nu$ is a finite subset of $H:=H(\Om,z):=Z\cap
\Om$.

In view of Lemma \ref{l:FR} (i), to $\psi$ there exists another
sequence $\theta:\ZZ\to \CC$ of finite support $Q:=\supp
\theta$ such that $\psi=\theta\star\widetilde{\theta}$. Let us
define the measure $\sigma:=\sum_{k\in Q} \theta(k)
\delta_{kz}$. Note that $Q$, though, is not necessarily
included in $H$, therefore, the finite subset $\supp \s \subset
Z$ is not necessarily a subset of $\Om$. Nevertheless,
$\psi=\theta\star\widetilde{\theta}$ means that for each $k\in
\ZZ$ we have $\psi(k)=\sum_{m\in Q} \theta(m)
\widetilde{\theta}(k-m) = \sum_{m\in Q} \theta(m)
\overline{\theta(m-k)}$ and so this vanishes outside $S\subset
H$, whence
\begin{equation}\label{sigmastarsigmatilde}
\sigma\star\widetilde{\sigma} =\sum_{m\in Q} \sum_{j\in Q}
\theta(m) \overline{\theta(j)} \de_{mz}\de_{-jz} =\sum_{k\in
\ZZ} \left( \sum_{m\in Q} \theta(m)\overline{\theta(m-k)}
\right) \de_{kz} = \sum_{k\in S} \psi(k) \de_{kz}=\nu,
\end{equation}
which is supported in $S\subset Z\cap \Om$.

Now our construction is the following. For a compact
neighborhood $W$ of $0$ (to be chosen suitably later), the
function $g:=\chi_W\star\sigma$ is a compactly supported step
function, hence is in $L^2(\mu_G)$, moreover, it has converse
$\widetilde{g}=\widetilde{\chi_W\star\sigma}=\widetilde{\sigma}*\chi_{-W}$
and thus the "convolution square" $f:=g\star\widetilde{g}$,
positive definite and continuous according to Lemma
\ref{l:convolutionsquare}, will be just
\begin{align}\label{eq:THEfunction}
f:=g\star\widetilde{g}&=\chi_W\star\sigma\star\widetilde{\sigma}\star\chi_{-W}
=\chi_W\star\chi_{-W}\star\sigma\star\widetilde{\sigma}
=\chi_W\star\chi_{-W}\star\nu = \sum_{k\in S} f_k, \notag
\\ & \quad \textrm{with} \quad f_k(x):=\psi(k)~
\left(h\star\de_{kz}\right)(x) =\psi(k) ~h(x-kz) ,\quad
h:=\chi_W\star\chi_{-W},
\end{align}
using also \eqref{eq:deltastarphi}. Clearly $\supp f_k= \supp h
+ kz \subset W-W+kz$, which is a compact set itself in view of
compactness of $W$. So because of finiteness of $S$ we also
find that $\cup_{k\in S} (W-W+kz)$ is compact, whence $\supp f
\subset \cup_{k\in S} \supp f_k \subset \cup_{k\in S} (W-W+kz)$
shows that also $\supp f$ is compact.

If we choose now disjoint open neighborhoods $U_k\subset \Om$
of $kz\in \Om$ for each $k\in S$, by continuity of $(x,y)\to
x-y+kz$ from $G\times G \to G$ we can take a compact
neighborhood $W_k$ of $0$ with $W_k-W_k+kz \Subset U_k$, so
intersecting the finitely many $W_k$ for all $k\in S$ we arrive
at a $W^*:=\cap_{k\in S} W_k$, compact neighborhood of $0$,
such that $W^*-{W^{*}}+kz\Subset U_k\subset \Om$ ($\forall k\in
S$). Therefore if we chose some appropriate $W\Subset W^*$,
then also $\supp f \Subset \cup_{k\in S} U_k \subset \Om$. In
all, for any such choice of $W$ we arrive at $\supp f \Subset
\Om$, as needed. Our last condition on $W$ will be that we want
$kz \in W-W$ for a $k\in S$ only if $k=0$, i.e. we require
$W-W\cap \{kz~:~k\in S\} = \{0\}$. So it suffices to fix some
open neighborhood $V\subset G$ of $0$ such that $V\subset
G\setminus \{kz~:~k\in S, k\ne 0\}$, then choose a compact
neighborhood $W'\subset G$ of $0$ satisfying $W'-W' \subset V$
(which can again be done according to the continuity of
$(x,y)\to x-y$), and then take $W:=W'\cap W^*$.

So we arrive at $f\gg 0$, $f\in C_0(G)$, $\supp f \Subset \Om$,
with $\supp f \subset \cup_{k\in S} \supp f_k$ and $\supp
f_k\Subset (W-W+kz)$. It remains to compute the function values
of $f$ at $0$ and at $z$. First, as $\supp h \subset W-W \subset V$, $h$ vanishes on all $kz$ with $k\in S\setminus \{0\}$ by construction, so from $f(0)=\sum_{k\in S} \psi(k) h(-kz)$ we get $f(0)=\psi(0) h(0)= 1 \cdot \chi_W\star\chi_{-W}(0)= \mu_G(W)(>0)$, according to \eqref{eq:chistarchi}. Second, completely similarly we have $f(z)= \sum_{k\in S} \psi(k) h(z-kz)=\psi(1) \mu_G(W)$.

In all, we can take $F:=\frac{1}{\mu_G(W)} f$, which then has
$F(0)=1$, too, and hence $F\in \FF^c_G(\Om)$, moreover, $F(z) =
\psi(1)$, hence $|\psi(1)| \leq \CF^c_G(\Om,z)$. Having this
for all $\psi \in \FF^c_\ZZ(H)$, taking supremum yields $\CF^c
(H) \leq \CF^c_G(\Om,z)$, whence the theorem.
\end{proof}

\begin{proof}[Proof of Theorem \ref{th:KreneditsFini}]
The proof of the complex and real variants are almost identical
to the preceding one, once we carefully change all references
from $\ZZ$ to $\ZZ_m$, $\CF^c(H)$ to $\CF_m(H)$ and
$\FF^c_{\ZZ}(\Om,z)$ to $\FF_{\ZZ_m}(\Om,z)$, and using Lemma \ref{l:FR} (ii) instead of (i); and similarly in the real case, while noting that $Z:=\langle z \rangle$ is only
a finite subgroup with $Z\cong \ZZ_m$, so the natural
isomorphism $\eta(k):=kz$ acts between $\ZZ_m$ and $Z$ now.
However, here we don't have the equality of the extremal
quantities $\CF_m(H)$ and $\KK_m(H)$, as in this case only the
estimates of Proposition \ref{p:realpd} (iii) hold. Therefore,
the real and complex cases here split as formulated separately
in \eqref{eq:mainresultFini}.
We spare the reader from further details of the proof.
\end{proof}

\section{Final remarks}\label{sec:final}

In view of Theorems \ref{th:KreneditsL} and
\ref{th:KreneditsFini}, the connection between the real and
complex cases in $\ZZ$, found in Proposition \ref{p:realpd}
(ii) and (iii), extends to all LCA groups. That is, we obtain
\begin{corollary}\label{complexreal} Let $G$ be any locally
compact Abelian group and let $\Om\subset G$ be an open
(symmetric) neighborhood of $0$. Let also $z\in \Om$ be any
fixed point with $o(z)=m$, and denote $H_m(\Om,z):=\{k\in
\ZZ_m~:~ kz \in \Om\}$. Then we have
\begin{equation}\label{eq:realcomplexLCAm}
\cos(\pi/m) \CF_G(\Om,z)\leq \KK_G(\Om,z)\leq \CF_G(\Om,z).
\end{equation}
Also the "$m=\infty$" case holds true giving for torsion-free
elements $z\in\Om$ the equality
\begin{equation}\label{eq:realcomplexLCA}
\KK_G(\Om,z)=\CF_G(\Om,z).
\end{equation}
\end{corollary}

Let us recall that when ${\mathcal M}(H)$ or $\CF(H)$ is
known for a certain $H\subset \ZZ$, then further cases can be
obtained via the following duality result.
\begin{lemma}
\label{dualitylemma} {\bf (see \cite{R1}).} Let $H\subseteq\ZZ$
be arbitrary with $\{-1,0,1\}\subset H$. Then denoting
$H^{\star}:=(\NN\setminus H)\cup \{-1,0,1\}$ we have $ \MM(H)
\MM(H^{\star})=2$.
\end{lemma}
\begin{remark} By Proposition \ref{p:realpd} (i), it is equivalently
formulated as $\KK(H)\KK(H^{\star})=\dfrac12$.
\end{remark}
\begin{remark}
The analogous finite dimensional duality relation in $\ZZ_m$ is
much easier, essentially trivial to obtain along the lines of
\cite{R1}. It gives with $H_m^{\star}:=(\ZZ_m\setminus H)\cup
\{-1,0,1\}$
$$
\MM_m(H) \MM_m(H^{\star})=2 \qquad {\rm or, ~ equivalently} \qquad
\KK_m(H)\KK_m(H^{\star})=\frac12,
$$
taking into account Proposition \ref{p:realpd} (i) again.
\end{remark}

It is explained in \cite{kolountzakis:pointwise} in the context
of the groups $\RR^d$ and $\TT^d$ that description of the
$\CF_G(\Om,z)$ problems by the Carath\'eodory-Fej\'er type
extremal problems on $\ZZ$ or $\ZZ_m$ automatically extends
these duality results to the more general situation. Regarding Problem
\ref{pr:ptwgeneral} we have
\begin{corollary}
For any LCA group $G$, open set $\Omega\subseteq G$ and
$z\in\Omega$ we have $\KK_G(\Om)\KK_G(\Om^{\star})=\frac12$ where
$\Om^*$ is any symmetric open set with $Z \cap \Om \cap \Om^* =
\{0,z,-z\}$ and $(\Om \cup \Om^*) \supset Z$, with $Z:=\langle
z \rangle $, i.e. $\{kz ~:~ k\in \ZZ\}$ or $\{kz ~:~ k\in
\ZZ_m\}$, respectively.
\end{corollary}
\begin{remark}\label{rem:cplxdualitymayfail} When
$\CF_G(\Om,z)=\KK_G(\Om,z)$ and
$\CF_G(\Om^{\star},z)=\KK_G(\Om^{\star},z)$--in particular when
$o(z)=\infty$--then we have the analogous formula
$\CF_G(\Om,z)\CF_G(\Om^{\star},z)=2$ for the complex
quantities. However, it fails whenever
$\CF_G(\Om,z)=\KK_G(\Om,z)$ or
$\CF_G(\Om^{\star},z)=\KK_G(\Om^{\star},z)$ does so.
\end{remark}

This of course covers the corresponding results for $\RR^d$ and
$\TT^d$ given in \cite[Corollary 4.7]{kolountzakis:pointwise}.

It would be interesting -- perhaps by a direct argument
extending that in \cite{R1} -- to derive this duality result
without relying on Theorems \ref{th:KreneditsFini} and
\ref{th:KreneditsL}.




\end{document}